\documentclass[12pt]{amsart}
\usepackage{amsmath}
\usepackage{hyperref}
\usepackage[margin=1in]{geometry}

\usepackage{graphics}

\usepackage{amssymb}
\usepackage{fourier}
\usepackage[T1]{fontenc}

\allowdisplaybreaks

\usepackage[
backend=biber,
style=alphabetic, url=false
]{biblatex}

\usepackage{physics}

\usepackage{amsthm}

\usepackage{graphicx}
\usepackage{subcaption} 
\usepackage{euscript}
\usepackage{multicol}
\usepackage[dvipsnames]{xcolor}

\newcommand{\N}{\mathbb{N}}
\newcommand{\R}{\mathbb{R}}

\newcommand{\B}{\mathbb{B}}
\newcommand{\D}{\mathbb{D}}
\newcommand{\Sph}{\mathbb{S}}

\newcommand{\sff}{\mathrm{I\!I}}

\newcommand{\Ric}{\text{Ric}}

\newcommand{\dive}{\text{div}}
\newcommand{\p}{\partial}

\newtheorem{theorem}{Theorem}[section]
\newtheorem{lemma}[theorem]{Lemma}
\newtheorem{corollary}[theorem]{Corollary}
\newtheorem{definition}[theorem]{Definition}
\newtheorem{remark}[theorem]{Remark}
\newtheorem{example}[theorem]{Example}
\newtheorem*{theorem*}{Theorem}

\numberwithin{equation}{section}

\newenvironment{nouppercase}{%
  \renewcommand{\uppercasenonmath}[1]{}}{}

\begin{document}
	

 \title[CAPILLARY]{Capillary Surfaces in Manifolds with Nonnegative Scalar Curvature \\and Strictly Mean Convex Boundary
}

\author[]{Yujie Wu}
\address{ Department of Mathematics, Stanford University, Building 380, Stanford, CA 94305,
USA}
\email{yujiewu@stanford.edu}

	\begin{abstract}
		In this paper we use stable capillary surfaces (analogous to the $\mu$-bubble construction) to study manifolds with strictly mean convex boundary and nonnegative scalar curvature. 
        We give an obstruction to filling 2-manifolds by such 3-manifolds based on the Urysohn width. We also obtain a bandwidth estimate and establish other geometric properties of such manifolds. 
	\end{abstract}
	
	\begin{nouppercase}
        \maketitle
    \end{nouppercase}

\section{Introduction}

In \cite{gromov2018scalar}, Gromov asked the question of finding sufficient conditions to allow or disallow fill-in of a given Riemannian manifold $Y^n $ as the boundary of a Riemannian manifold $X^{n+1}$ with nonnegative scalar curvature.  Can the mean curvature of $Y =\p X$ prove enough influence so that the we cannot prescribe certain geometry properties on $X$?

In this paper we  are interested in the case when $n=2$. If $Y$ is a connected orientable closed surface with positive  Gaussian curvature, then there is an isometric embedding of $Y$ into $\R^3$ as a strictly convex surface, we denote the mean curvature of this embedding as $H_0$ (such embedding is unique up to isometry of $\R^3$).   Using this, Shi-Tam \cite{shi2002positive}  proved that if there is some $(X,g)$ with nonnegative scalar curvature filling such a $Y$ with positive mean curvature $H$, then
\begin{equation*} 
    \int_{Y} H d\sigma \leq \int_{Y} H_0 d\sigma,
\end{equation*}    
where $d\sigma$ is the volume form induced from the metric $g$. Moreover, equality holds if and only if $X$ is a domain in $\R^3$. This result gives a positive answer to the question in the first paragraph.

In this paper, we give another answer in the following theorem using capillary surfaces. As an example, if a 3-manifold fills $\Sph^2$ with $H\geq 2$ and has nonnegative scalar curvature, then the first Urysohn width of $\Sph^2$ in the induced metric is no more than $4.5\pi$ (take $a_0=\frac{2}{3}, d_0=\frac{3\pi}{4}$).   Recent results in upper bound of Urysohn width  was also obtained in \cite{wang2023uryson}, \cite{liokumovich2023waist}.
\begin{theorem}\label{fillin}
    If $(N^3,\p N,g)$ is a compact manifold with simply-connected boundary such that $R_g \geq 0$ and $H_{\p N} \geq \frac{\pi}{d_0}+a_0$, then the first Urysohn width of $\p N$ (with respect to the induced metric $g$) is bounded: $U_1(\p N)\leq 4d_0+\frac{\pi}{a_0}$.      
\end{theorem}

In another direction, Gromov proved the following bandwidth estimate, here we denote $T^{n-1}$ as the $n-1$ dimensional torus. The theorem can be proved using the idea of $\mu$-bubbles \cite{gromov2019four}. 
\begin{theorem*} [\cite{gromov2018metric}]
    Let $M_0 =T^{n-1}\times [-1,1], \p_{\pm} M_0=T^{n-1} \times \{\pm 1\}$, if a Riemannian manifold $(M^n,g) , 2\leq n\leq 7$ admits a continuous map $f:(M,\p_{\pm} M) \rightarrow (M_0,\p_{\pm} M_0) $ with nonzero degree and scalar curvature bounded from below $R_g \geq n(n-1)$, then the distance of $\p_{+} M$ and $\p_- M$  is bounded: $\text{dist}_g(\p_+ M,\p_-M)\leq \frac{2\pi}{n}$.   
\end{theorem*}

Using $\mu$-bubbles, that is, studying stable hypersurfaces with prescribed mean curvature in a manifold with positive scalar curvature (PSC) has given fruitful results in recent years (see for example \cite{li2020polyhedron} \cite{zhu2021width}\cite{chodosh2024generalized}\cite{chodosh2022complete}). 
In these works, the fact that the scalar curvature of the manifold has to obtain a strictly positive lower bound is crucial.
On the other hand, for manifolds with boundary, to constrain a minimizer of a generalized area functional, we can prescribe both mean curvature and the angle of intersection along the boundary. In particular, mean convexity assumption is helpful to constrain capillary surfaces. This allows us to relax the assumption of PSC when studying manifolds with boundary.

Gromov studied a band with PSC of the form $\Sigma^2 \times [-1,1]$, for a closed surface with $\chi_{\Sigma}\leq 0$ (for example the torus).  Our model example in the case $\Sigma$ has boundary would be $M_0=(\Sigma, \p \Sigma)\times [-1,1] $, with a Riemannian metric such that $R_{M_0}\geq 0$, and the boundary $\p M_0$ is strictly mean convex. We will use capillary surfaces in $M_0$ as an analogy to the $\mu$-bubbles in Gromov's bandwidth estimate.  To avoid regularity problem that may occur when the capillary surface touches the corner $\p \Sigma \times \{\pm 1\}$, we will cap the model example $M_0$ with a top $M_{+}$ and a bottom $M_-$. This explains the decomposition we have in the following theorem.

\begin{theorem}[Bandwidth estimate] \label{bandwidthest} 
    Consider a compact 3-manifold $(M,\p M)$ with the following decomposition (see Figure 1), 
    $$M=M_0 \cup M_{\pm},\,\, \p M=\p_0 M \cup \p_{\pm} M,\,\, \p M_{\pm}= S_{\pm} \cup \p_{\pm} M,\,\, \p M_0= \p_0 M \cup S_{\pm}.$$ 
    Here $M_0$ is homeomorphic to $S_- \times [-1,1]$,  $S_-$ is an orientable surface with boundary and Euler characteristic $\chi(S_-) \leq 0$. Assume $\p_+ M \cap \p _- M =\emptyset$, and for any component $\Gamma_{+}$ of $\p_{+} M$, there is a component $\Gamma_{-}$ of $\p_{-} M$ such that, $d_{\p M} (\Gamma_+,\Gamma_-) <\infty$. 

    We denote the scalar curvature of $M$ by $R_M$, mean curvature of $\p_i M$ by $H_i$ for $i \in \{+,-,0\}$.     
    If $R_M \geq 0, H_0\geq 2, H_{\pm} \geq 0$, then $d_{M}(S_+,S_-)\leq 4+\pi.$  
\end{theorem}

\begin{figure}
    \def\svgwidth{2in}
    \begin{center}
\begingroup%
  \makeatletter%
  \providecommand\color[2][]{%
    \errmessage{(Inkscape) Color is used for the text in Inkscape, but the package 'color.sty' is not loaded}%
    \renewcommand\color[2][]{}%
  }%
  \providecommand\transparent[1]{%
    \errmessage{(Inkscape) Transparency is used (non-zero) for the text in Inkscape, but the package 'transparent.sty' is not loaded}%
    \renewcommand\transparent[1]{}%
  }%
  \providecommand\rotatebox[2]{#2}%
  \newcommand*\fsize{\dimexpr\f@size pt\relax}%
  \newcommand*\lineheight[1]{\fontsize{\fsize}{#1\fsize}\selectfont}%
  \ifx\svgwidth\undefined%
    \setlength{\unitlength}{258.25174407bp}%
    \ifx\svgscale\undefined%
      \relax%
    \else%
      \setlength{\unitlength}{\unitlength * \real{\svgscale}}%
    \fi%
  \else%
    \setlength{\unitlength}{\svgwidth}%
  \fi%
  \global\let\svgwidth\undefined%
  \global\let\svgscale\undefined%
  \makeatother%
  \begin{picture}(1,1.08848544)%
    \lineheight{1}%
    \setlength\tabcolsep{0pt}%
    \put(0,0){\includegraphics[width=\unitlength,page=1]{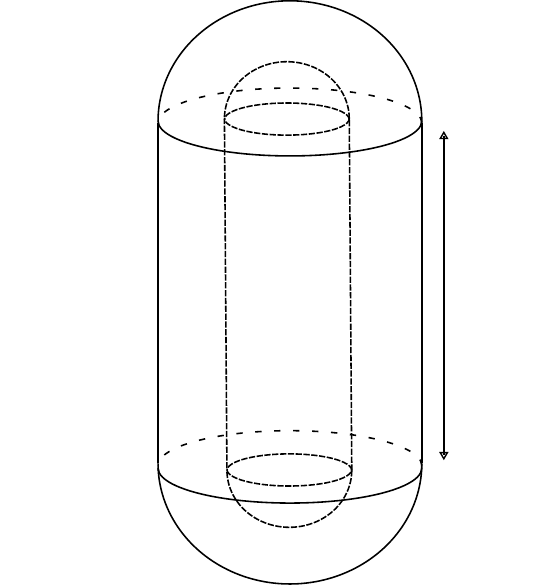}}%
    \put(0.87604712,0.96098929){\color[rgb]{0,0,0}\makebox(0,0)[lt]{\lineheight{1.25}\smash{\begin{tabular}[t]{l}$M_+$\end{tabular}}}}%
    \put(0.88253598,0.54276842){\color[rgb]{0,0,0}\makebox(0,0)[lt]{\lineheight{1.25}\smash{\begin{tabular}[t]{l}$M_0$\end{tabular}}}}%
    \put(0.88733459,0.10424102){\color[rgb]{0,0,0}\makebox(0,0)[lt]{\lineheight{1.25}\smash{\begin{tabular}[t]{l}$M_-$\end{tabular}}}}%
    \put(-0.00289281,0.8501527){\color[rgb]{0,0,0}\makebox(0,0)[lt]{\lineheight{1.25}\smash{\begin{tabular}[t]{l}$S_+$\end{tabular}}}}%
    \put(0.00541972,0.20161388){\color[rgb]{0,0,0}\makebox(0,0)[lt]{\lineheight{1.25}\smash{\begin{tabular}[t]{l}$S_-$\end{tabular}}}}%
    \put(0,0){\includegraphics[width=\unitlength,page=2]{drawing.pdf}}%
  \end{picture}%
\endgroup%

    \end{center}
    \caption{An example of $M$ when $S_-$ is an annulus.} \label{thm1.3}
\end{figure}

In \cite{ros1997stability}, Ros and Souam studied constant mean curvature surfaces which intersect the ambient boundary at a fixed angle, called capillary surfaces. 
In the second variation formula for capillary surfaces, instead of fixed angle, we allow the prescribed angle to vary.
For manifolds with nonnegative scalar curvature and strictly mean convex boundary, 
we are able to gain rigidity results for 2-manifolds (Theorem \ref{2d-cpt}), and  exhaustion of the boundary for 3-manifolds (Theorem \ref{3d-decomp}).
\begin{theorem}\label{2d-cpt}
    Consider a connected Riemannian manifold $\Sigma^2$ with boundary, such that   $R_{\Sigma}\geq 0, k_{\p \Sigma} \geq 1$, then $\p \Sigma$ is connected  with length no more than $2\pi$; for any $x\in \Sigma$,  $d_{\Sigma}(x,\p \Sigma) \leq 1,$ so $\Sigma$ is compact. $\Sigma$ is a topological disk.  If $|\p \Sigma|=2\pi$, then $\Sigma$ is isometric to a flat unit disk in $\R^2$.   
\end{theorem}

The fact that every point in this manifold must be at most distance 1 away from the boundary follows from Li-Nguyen \cite{li2014compactness} and Li \cite{li2014sharp}. Furthermore, Li \cite{li2014sharp} conjectured the following: A complete Riemannian manifold $M^n$ with nonnegative Ricci curvature and $A_{\p M}\geq 1$ must be compact ($A_{\p M}$ is the second fundamental form of the boundary).
This theorem solves the 2-dimensional case of this conjecture. We note in Theorem \ref{2d-cpt} we do not assume orientability or embeddedness. In the case of $M^n\subset \R^{n+1}$, a strictly convex hypersurface bounding a region with pinched second fundamental form, Hamilton \cite{hamilton1994convex}  showed that $M$ must be compact.   

Contrary to the 2-dimensional case, for a connected 3-manifold with nonnegative scalar curvature and $H_{\p M} \geq 2$ ($H_{\p M}$ is mean curvature of the boundary), its boundary might not be connected or compact. But we can obtain an exhaustion of the boundary as below, and show that the boundary has linear volume growth under further topological assumptions (see Corollary \ref{3d-growth}).

\begin{theorem}\label{3d-decomp}
    Let $(M^3, \p M)$ be a complete connected Riemannian manifold with nonnegative scalar curvature, and $\p M$ is connected non-compact with $H_{\p M}\geq 2$. 
    
    We can find an exhaustion of the boundary $\p M=\cup_k Z_k$, with $\p Z_k =\p \Sigma'_k$, where $\Sigma'_k$ is a union of finitely many capillary surface of disk type in $M$. Each component $\Sigma_k$ of $\Sigma'_k$ has bounded boundary length: $|\p \Sigma_k| \leq 2\pi$, and $d_{\Sigma_k}(x,\p \Sigma_k) \leq 2$ for any $x\in \Sigma_k$ . 
        
    Furthermore, if $\pi_1(\p M)=0$, let   $E_k$ be an unbounded component of $\p M \setminus Z_k$ such that $E_{k+1} \subset E_k$ (so $(E_k)_{k=1}^{\infty}$ is an end of $\p M$), then $\overline{E_{k}}\setminus E_{k+1}$ is connected and $\sup_k \text{diam}_{\p M}(\overline{E_{k}}\setminus E_{k+1}) \leq 5\pi$.       
\end{theorem}

Using Theorem \ref{3d-decomp} when $\pi_1(\p M)=0$ and $\p M$ has weakly bounded geometry we can apply section 3 in \cite{wu2023free} to get the following corollary.  
\begin{corollary}\label{3d-growth}
    Let $M^3$ be a complete connected Riemannian manifold with nonnegative scalar curvature, whose non-compact boundary $\p M$  is simply connected, uniformly mean convex and has weakly bounded geometry, then each end of $\p M$ has linear volume growth. In particular, if $\p M$ has finitely many ends, then  $\p M$ has linear volume growth.  
\end{corollary}

\textbf{Acknowledgements.} The author wants to thank Otis Chodosh for his continuous support and encouragement, and for many helpful discussions. The author thanks Chao Li for discussions over the fill-in problem and encouragements through improvements of the first draft. The author thanks Brian White for discussions over the maximum principles and regularity.  The author thanks Shuli Chen and Jared Marx-Kuo for interest in this work.

\section{2D results}\label{2d}

Throughout this paper, if $(N^{k+1},\p N,g), k\geq 1,$ is a Riemannian manifold with boundary, we denote the second fundamental form of $\p N \subset N$ with respect to the outward pointing unit normal $\nu$  as $\sff(X,Y)=-\langle \nabla_X Y, \nu \rangle$ for vector fields $X,Y$  tangent to $\p N$. Then the scalar-valued mean curvature is written as $H=\sum_{i=1}^k -\langle \nabla_{e_i} e_i,\nu\rangle$ for an orthonormal basis $(e_i)_i^k$ of $T_p \p N$  at some point $p\in \p N$. In this convention, the boundary of the unit ball $\B^3$  in $\R^3$ has positive mean curvature, $H_{\p \B^3}=2$.

\begin{proof}[Proof of Theorem \ref{2d-cpt}]
    The $d_{\Sigma}(x,\p \Sigma) \leq 1$ part follows directly from \cite[Proposition 2.1]{li2014compactness}  and \cite[Theorem 1.1]{li2014sharp}. Note that from this result we know that $\Sigma$ is compact if and only if $\p \Sigma$ is compact.  

    \textit{Step 1.} If we know that $\Sigma$ is compact,   we can show that $\p \Sigma$ is connected and the length of $\p \Sigma$ is no more than $2\pi$.

    We apply Gauss Bonnet theorem to get, denoting $\frac{1}{2}R_{\Sigma}=K_{\Sigma}$: 
    \begin{equation}\label{2dGB}
        2\pi \geq 2\pi \chi=\int_{\Sigma} K_{\Sigma}+\int_{\p \Sigma} k_{\p \Sigma}\geq |\p \Sigma|>0.
    \end{equation}
    Here we used $\Sigma$ has at least one boundary component, so $\chi_{\Sigma}\leq 1$. And if $\p \Sigma$ has more than 1 boundary component then we would get $\chi_{\Sigma}\leq 0$, a contradiction.
   So $\p \Sigma$ is connected with $|\p \Sigma|\leq 2\pi$, and $\chi(\Sigma)>0$, so $\chi_{\Sigma}=1$, and $\Sigma$ is a topological disk.

    \textit{Step 2.}
    We now look at the rigidity statement under the assumption that $\Sigma$ is compact (and $\p \Sigma$ is connected by Step 1).

    If $K\geq 0,k\geq 1$ and $|\p \Sigma|=2\pi$, then from (\ref{2dGB}) we must have $K=0, k=1$ everywhere. So $\Sigma$ is locally isometric to the flat unit disk $\D=\{(x_1,x_2), x_1^2+x_2^2\leq 1\}$ , we show that the local isometry can be extended globally.  
    
    We denote the unit speed loop of $\p \Sigma$ as $\gamma:[0,2\pi] \rightarrow\p \D$, and write $\gamma(0)=\gamma(2\pi)=p$. Then by choosing $\epsilon$ small enough, we can find an isometry for each $t\in[0,2\pi]$ with  $\Psi_{t}:B_{\epsilon}(\gamma(t))\rightarrow \D$. Combining $\Psi_t$ with a rotation of $\D$ around origin, we can patch up these isometry $\Psi_t$ to get a global isometry $\Psi_0$  from $B_{\epsilon}(\p \Sigma)$ to a neighborhood of $\p \D$  in $\D$. 
    
    We now define a global isometry $\Psi$ from $\Sigma$ to $\D$  as follows, for $x\in B_{\epsilon}(\p \Sigma)$, $\Psi(x)=\Psi_0(x)$. Now fix $p\in \p \Sigma$ and for any $q\in \Sigma \setminus B_{\epsilon}(\p \Sigma)$, for any path $s(t):(0,1] \rightarrow \Sigma^{\circ}$ with $s(0)=p,s(1)=q$. We cover $s(t)$ with interior balls each of which can be mapped isometrically into an interior ball in $\D$, in a way that agrees on the overlap (see \cite[Theorem 12.4]{lee2018introduction}). Then applying in \cite[Corollary 12.3]{lee2018introduction}, we have obtained a global isometry from $\Sigma$ to $\D$.

    \textit{Step 3.}     
    We now show $\p \Sigma$ is connected (if $\Sigma$ is  non-compact). 
    
    If not, then there are at least two components $\p_1 \Sigma, \p_2 \Sigma$ of $\p \Sigma$. Let's first assume that $\alpha=\inf\left\{d(x,y),x\in \p_1 \Sigma, y \in \p_2 \Sigma\right\}$ can be obtained by a stable free boundary geodesic $\beta_0(t)$. Let $T$ be a unit normal vector field of $\beta_0$, notice that along $\p \Sigma$, $T$ is tangent to $\p \Sigma$, this vector field generates a local variation of $\beta_0$ denoted by $\beta_s(t)$.  Then the second variation formula of length $L(\beta_s)$ implies,  
    \begin{equation*}
        0\leq  \frac{d^2}{ds^2}\Bigg\rvert_{s=0}L(\beta_s)=\int_{\beta_0} -\frac{R_{\Sigma}}{2}+ \int_{\p \beta_0} \langle \nabla_{T} T, -\beta_0' \rangle \leq \langle \nabla_{T} T, -\beta_0'(0) \rangle+\langle \nabla_{T} T, \beta_0'(\alpha) \rangle\leq -2 <0,
    \end{equation*}   
    where we used convexity at the end points, i.e. $\langle \nabla_T T,\beta_0'(0)\rangle=k_{\p\Sigma}(\beta_0(0))\geq 1$ and $\langle \nabla_T T,\beta_0'(\alpha)\rangle=-k_{\p\Sigma}(\beta_0(\alpha))\leq -1$; this is a contradiction.
    
    Now we look at the case $\alpha$ can't be realized by a stable free boundary geodesic. In this case, at least one of the component is not compact.   
    Assume $\p_1 \Sigma$ is non-compact, then we can try to find a minimizing geodesic line via points on $\p_1 \Sigma$. To elaborate, consider the universal cover of $\Sigma$ named $\tilde{\Sigma}$, then we have again (at least) two components of the boundary $\p \tilde{\Sigma}$, denoted $ \p_1 \tilde{\Sigma} \cup \p_2 \tilde{\Sigma}$ 
    and $\p_1 \tilde{\Sigma}$ is not compact. We denote $\p_1 \tilde{\Sigma} \cong \R$ as $\gamma(t)$, with $|\gamma'(t)|=1$, $\gamma(0)=q$, $\gamma(\pm n)=\pm q_n$. Consider the minimizing geodesic in $\tilde{\Sigma}$ from $+q_n$ to $-q_n$ denoted  $l_n(t)$  , which cannot touch the boundary except at end points due to strict convexity. 
    
    We fix a point $ p\in \p_2 \tilde{\Sigma}$ and  a minimizing path from $p$ to $q$,  denoted by $\overline{pq}(t)$. Now any minimizing geodesic in $\tilde{\Sigma}$  between $\pm q_n$ must intersect $\overline{pq}$. Indeed, if there is a minimizing geodesic $l$ between $\pm q_n$ and $l\cap \overline{pq}=\emptyset$, we can minimize distance of $+ q_n$ (similarly for $-q_n$) to $\overline{pq}$, and then use this to build a path transversal to $\overline{pq}$ with intersection number 1. Concatenating with $l,$ we get a loop in $\tilde{\Sigma}$ whose intersection number with $\overline{pq}$ is equal to 1. 
    But this loop is homotopic to the constant loop at e.g. $+q_n$ by simply-connectedness, a contradiction (see section 2.4 in \cite{guillemin2010differential}).

    Now consider a minimizing geodesics $\gamma_n(t)$ connecting $\pm q_n$. First, $|\gamma_n| \rightarrow \infty$, since we can pick $\pm q_n$ to be in $\tilde{\Sigma} \setminus B_{n}(\overline{pq})$ for large $n$.  We look at one of the intersection points of $\gamma_n$ and $\overline{pq}$ named $s_n$ and the velocity vector of $\gamma_n$ at $s_n$ named $v_n$, then $\{(s_n, v_n), n \in \N\}$ is contained in a compact set, and we have a subsequence converging to some point $(s,v)$. The geodesic ray from $s$ with velocity $v$ (respectively $-v$) must have infinite length because we can get $C^1_{\text{loc}}$-convergence
    of these geodesics when $(s_n,v_n) \rightarrow (s,v)$.
    This means that we have found a geodesic line. 

    Now we can use the Toponogov's splitting theorem generalized to manifolds with convex boundary (\cite[Theorem 5.2.2]{burago1977convex}) to conclude that $\tilde{\Sigma} = \R \times I$ where $I$ is a connected one manifold with boundary, this is  a contradiction to the boundary being strictly convex.

    \textit{Step 4.}   We assume $\Sigma$ is not compact and get a contradiction. In particular, we assume the diameter of $\p\Sigma$ is unbounded. 
    By the result of Step 3, we can assume that $\p \Sigma$ is connected. The idea is that if the diameter of the boundary $\p \Sigma$ is large, then we can prescribe a ``capillary'' minimizing geodesic, and use stability inequality to get a contradiction. 
    
    The set up in Step 4 and Step 5 are written in general, not restricting to the ambient manifold being two dimensional.
 We use the notation $B_r(x)$ to denote a ball  of radius $r$ around $x$ using the distance function on $\Sigma$, and $B'_r(x) \subset \p\Sigma$ is a ball of radius $r$ around $x\in \p \Sigma$ using the distance function $\p \Sigma$ with respect to the induced metric.

    Assume $w:\p\Sigma \rightarrow[-1,1]$ is the following Lipschitz function, for some fixed $x_0 \in \p\Sigma$,
\begin{equation*}
w(x) = \begin{cases}
-1 & x\in B_2'(x_0)\\
\cos\rho(x) & -1<w(x)<1 \\
1 & x\in \p\Sigma \setminus B_{2\pi}'(x_0)
\end{cases}
\end{equation*}

    Here   $\rho(y):\p \Sigma\rightarrow \R$ is a smooth function with$|\nabla \rho|<1$. We denote $W_{\pm}:=\{x\in \p\Sigma: w(x)=\pm 1\}$ 
    and require the set $\{x\in \p \Sigma: \rho(x)=\pi\}=\p W_-$ and the set $\{x\in \p \Sigma: \rho(x)=0\}=\p W_+$ to be smooth submanifolds.
    
     Now we minimize the functional in (\ref{capigeo}) among open sets with finite perimeter, containing the set $B'_{2}(x_0)$.
    Take such a Caccioppoli set $\Omega$ and let, 
    \begin{equation}\label{capigeo}
        A(\Omega)=\mathcal{H}^1(\p \Omega) +\int_{\Omega \cap \p \Sigma} w,
    \end{equation}
    where $\mathcal{H}^1(\p\Omega)$ is the perimeter of $\p \Omega$ in $\Sigma$.

    We note that if a smooth minimizer $\p \Omega \neq \emptyset$ exists, then for each component $\gamma$ of $\p \Omega$ with $\p \gamma \neq \emptyset$, using the second variation formula \cite{ros1997stability}, we have over the minimizing geodesic $\gamma$, 
\begin{align*}
    A''(\gamma)=&\int_{\gamma} -\phi\Delta_{\gamma} \phi -(\Ric(N,N)+|\sff_{\gamma}|^2)\phi^2 \\
    &+\int_{\p \gamma} \frac{\phi^2}{\sin\rho}\langle  \nabla_{\bar{\nu}} \bar{\nu},\bar{N}\rangle+\frac{\phi^2}{\sin^2 \rho} \nabla_{\bar{\nu}} w +\phi \nabla_{\nu} \phi + \phi^2 \cot \rho \langle N,\nabla_{\nu} \nu \rangle,
\end{align*}
where $\overline{N}$ (respectively $N$) is the outward unit normal of $\p \Sigma \subset \Sigma$ (respectively $\p \Omega \subset \Omega$), and $\overline{\nu}$ (respectively $\nu$) is the outward unit normal of $\p\Omega\cap \p \Sigma \subset \Omega \cap \p \Sigma$ (respectively $\p \gamma \subset \gamma$).     

We can plug in $\phi=1$, use $\sff_{\gamma}=0=\langle N ,\nabla_{\nu} \nu \rangle$, and  use the Gauss Equation,
\begin{align*}
    0&\leq A''(\gamma)=\int_{\gamma} -\frac{1}{2}R_{\Sigma}+\int_{\p\gamma}\frac{1}{\sin \rho}(\langle  \nabla_{\bar{\nu}} \bar{\nu},\overline{N}\rangle-\nabla_{\nu} \rho)\\
     &< \int_{\p \gamma} \frac{1}{\sin\rho}\left( -1+1 \right) =0,
\end{align*}
using $|\nabla \rho|< 1$ and $-k_{\p \Sigma}=\langle  \nabla_{\overline{\nu}} \overline{\nu},\overline{N}\rangle\leq -1$, leading to a contradiction.

\textit{Step 5.} In this step we write down some technical details needed to show that a smooth minimizer exists (used in Step 4). 
    
If $\Omega$ is a candidate in a minimizing sequence of $A$, by assumption  we have $\p\Omega \cap B'_{2}(x)=\emptyset$. 
In fact we can  find an open neighborhood $\Omega'$ of $W_-=\{x\in\p\Sigma: w(x)=-1\}$ in $\Sigma$, so that any minimizing sequence must contain $\Omega'$. 

To elaborate, we assume without loss of generality each $\Omega$ in the minimizing sequence has smooth boundary, now
we show that for some choice of $\Omega'$, 
    $$\delta(\Omega):=A(\Omega\cup \Omega')-A(\Omega)=\mathcal{H}^1(\p \Omega' \setminus \Omega)-\mathcal{H}^1(\p \Omega \cap \Omega')+\int_{\p \Sigma \cap \left(\Omega'\setminus\Omega\right)}w\leq 0.$$

We consider the following family $\Phi_t(x):=\exp(-\varphi_t(x)\overline{N}(x))$; recall $\overline{N}$ is the outward pointing unit normal of $\p \Sigma \subset \Sigma$. Here $\varphi_t(x):=\max\{s\phi(x)+t,0\}$ for some small $s>0$ to be chosen, and $\phi:\p\Sigma \rightarrow [-1,1]$ is a smooth function, such that $\{x\in \p \Sigma: \phi(x)>0\}=W_-^{\circ}$, and $\nabla \phi(x)\neq 0$ for any $x\in \p W_-$. Note that $\Gamma_t:=\overline{\Phi_t(\p \Sigma) \setminus \p \Sigma}$ is a smooth submanifold in $\Sigma$, and as $s,t \rightarrow 0$, $\Gamma_t$ converges to a smooth domain in $\p \Sigma$. We denote the unit normal of $\Gamma_t$ by  $\nu_t$,  pointing in the direction as $t$ increases. Then since $\nabla \phi(x)\neq 0$ for any $x\in \p W_-$, we have $\nu_0(x)\cdot \overline{N}(x)>-1=w(x)$. We pick $s'$ small enough, so that for any $t\in[-s,s']$ and $x\in \Gamma_t\cap \p \Sigma$, $\nu_t(x)\cdot \overline{N}(x)\geq w(x)$. Also because $\Gamma_t$ converges to a smooth domain in $\p \Sigma$ as $s,t \rightarrow 0$, using $k_{\p \Sigma}\geq 1$, we know that $\dive_{\Gamma_t} (\nu_{t})\leq -0.5$.


    Let $\Omega':=\cup_{t\in[-s,s']} \Gamma_t$ be the union of these ``foliation'', containing a tubular neighborhood of $W_-$. 
    By divergence theorem for Lipschitz domains we have,
    \begin{align*}
        \delta(\Omega)&\leq\int_{\p \Omega' \setminus \Omega} \nu_t \cdot \nu_{\p \Omega'}-\int_{\p \Omega \cap \Omega'}\nu_t\cdot \nu_{\p \Omega} +\int_{\p \Sigma \cap \left(\Omega'\setminus\Omega\right)}w\\
        &=\int_{\Omega'\setminus\Omega}\dive_{\Gamma_t}(\nu_t)+\int_{\p \Sigma \cap \left(\Omega'\setminus\Omega\right)} \nu_t \cdot (-\overline{N}) + w\\
        &\leq \int_{\Omega'\setminus\Omega}\dive_{\Gamma_t}(\nu_t) \leq 0
    \end{align*}
    where we used $\nu_{t}\cdot \overline{N}\geq w(x)$ for any $x\in \Gamma_t\cap \p\Sigma$ and $\dive_{\Gamma_t}(\nu_t)\leq -0.5$ for $t\in[-s,s']$. We get that any minimizing sequence must (eventually)  contain  $\Omega'$.

    We obtained for any minimizing sequence $\gamma_i =\p \Omega_i$, we have  $\gamma_i \supset \Omega' \supset   W_-$ eventually. A similar argument shows that $\gamma_i \cap W_+ =\emptyset$, so the term $\int_{\p \Sigma \cap \Omega_i} w$ is uniformly bounded for a minimizing sequence,  so $\mathcal{H}^1(\gamma_i)$  is also uniformly bounded.  We get that the minimizing sequence must be contained in a bounded set from the point $x$.
    We can continue the minimization with standard BV compactness and regularity theory (see \cite{philippis2015regularity}).

    We now check that the smooth minimizer $\gamma=\p \Omega$  must have nonempty boundary. Note $B'_{2}(x)  \subset\Omega'\subset \Omega$, and $\Omega$ is disjoint from the set $W_+$.   So if $\gamma$ is a minimizer with empty boundary, then $\Omega \cap \p \Sigma$ and $\p \Sigma \setminus {\Omega}$ is two disjoint nonempty open sets, whose union is $\p \Sigma$, a contradiction to $\p \Sigma$ being connected.
\end{proof}

\section{3D results}\label{3d}

Contrary to the 2-dimensional case, a 3-manifold with nonnegative scalar curvature and uniformly mean convex boundary, might not have connected boundary. We also might not have that any point is at bounded distance away from the boundary (using the same method would require non-negative Ricci curvature).

Before we start the proof of Theorem \ref{3d-decomp}, we need the following lemma which is analogous to Lemma 16 in \cite{chodosh2024generalized}. The inequality (\ref{PSCu}), when compared to the requirement of \cite{chodosh2024generalized} in the PSC setting, is suitable for nonnegative scalar curvature. And we added the assumption (\ref{BCu}), which is suitable for mean-convexity.

\begin{lemma} \label{DistBound}
    If there is a smooth function $u>0$ over a compact surface $(\Sigma,\p \Sigma)$ with nonempty boundary, such  that,
    \begin{align}
        \Delta_{\Sigma} u \leq \frac{R_{\Sigma}}{2} u +\frac{|\nabla_{\Sigma}u|^2}{2u} \quad &\text{over } \Sigma, \label{PSCu}\\ 
        \frac{\nabla_{\nu} u}{u} \geq a_0-k_{\p \Sigma} \quad &\text{over } \p \Sigma, \label{BCu}
    \end{align}  
    then $d(x,\p \Sigma) \leq \frac{2}{a_0}$ for any $x \in \Sigma$; the outward pointing unit normal along $\p \Sigma$ is denoted as $\nu$, $R_{\Sigma}$ is the scalar curvature of $\Sigma$ and $k_{\p \Sigma}$ is the geodesic curvature of $\p \Sigma$.     
\end{lemma}

\begin{proof}
    If not, assume $d(z,\p \Sigma )\geq \frac{2}{a_0}+2\delta$ for some $z \in \Sigma$ and $0<\delta<1$.   We then  find a minimizer of the following functional $\mathcal{F}(\Gamma)$  for sets with finite perimeters $\Gamma \subset \Sigma$ containing a neighborhood $U_0$   of $z$ (will be chosen later) and disjoint from $\p \Sigma$. Denote $\p \Gamma=\gamma$, and $\nu_{\gamma}$ the outward unit normal of $\gamma\subset \Gamma$,  
\begin{align*}
\mathcal{F}(\Gamma)=\int_{\gamma}u-\int_{\Sigma}hu(\chi_{\Gamma}-\chi_{\Gamma_0}),
\end{align*}  
here $h(y)$ is a mollification of the function $\hat{h}(y)= \frac{2}{\alpha-d_{\Sigma}(y,\p \Sigma)}>0$ when $x\in U$; here $\alpha=\frac{2}{a_0}+\delta$ and  $U:=\{y\in \Sigma,d(y,\p \Sigma) < \alpha\}$. We require $h\rvert_{\p \Sigma}<a_0, h\rvert_{\Sigma \setminus U}=\infty$, and $\frac{1}{2}h^2-|\nabla h| >0$ everywhere on $U$. We pick $\Gamma_0$ to be an open neighborhood of $z$ with smooth boundary and $h\rvert_{\Sigma \setminus \Gamma_0}\in L^{\infty}$. So for any smooth open set $\Gamma$, $\mathcal{F}(\Gamma)>-\infty$.

By the proof of Proposition 12 in \cite{chodosh2024generalized}, there is a smooth open neighborhood $U_0$ around $z$ such that  $\mathcal{F}(\Gamma\cup U_0)\leq \mathcal{F}(\Gamma)$  for any $\Gamma$ with smooth boundary and $h\rvert_{\Sigma \setminus U_0} \in L^{\infty}$.  This implies that any minimizing sequence must contain $U_0$ and $\inf \mathcal{F} > -\infty$.

We now check that any minimizing sequence must be disjoint from a fixed open neighborhood of $\p \Sigma$. So by interior regularity we have that a smooth minimizer exists \cite{tamaninni1984regularity} \cite{chodosh2024generalized}.

By first variation, if a smooth minimizer $\Gamma$ exists then its boundary $\gamma$ has geodesic curvature equal to $k_{\gamma}=h-\frac{\nabla_{\nu}u}{u}$. If $\gamma$ touches the boundary $\p \Sigma$ at some point $x$, then, 
\begin{align}\label{strongmaxi}
    k_{\gamma}(x)&=h\rvert_{\p\Sigma}(x)-\frac{\nabla_{\nu}u}{u} (x)\nonumber\\
    &\leq h\rvert_{\p\Sigma} (x)-a_0+k_{\p \Sigma}(x) <k_{\p \Sigma},
\end{align}
where we used (\ref{BCu}) and $h\rvert_{\p \Sigma}<a_0$.
Using this observation, similar to Step 5 in Theorem \ref{2d-cpt}, we  use a foliation along $\p \Sigma$ to modify our minimizing sequence so that it is disjoint from a fixed open neighborhood of $\p \Sigma$. 

To be precise, given any smooth minimizing sequence $\gamma_i=\p\Gamma_i$ , we show that $\mathcal{F}(\Gamma_i\setminus T') \leq \mathcal{F}(\Gamma_i)$, where $T'= \cup_{t\in[0,\epsilon],z\in \p \Sigma}\exp_z(-t\nu)$ for some small $\epsilon$ to be decided. If we fix $t\in [0,\epsilon]$, then we get $T'_t=\cup_{z\in\p\Sigma} \exp_z(-t\nu)$ a smooth curve with unit normal $\nu_t$ (pointing in the direction as $t$ increases), and that as $t \rightarrow 0, \nu_t \rightarrow -\nu$. We denote the outward pointing unit normal of $\p T' \subset T'$ and $\p \Gamma_i \subset \Gamma_i$ as $\nu_{\p T'}$ and $\nu_{\p \Gamma_i}$.  Then,
\begin{align*}
    \mathcal{F}(\Gamma_i\setminus T') - \mathcal{F}(\Gamma_i) &=\int_{\p T'\cap\Gamma_i} u+\int_{\p \Gamma_i \cap T'} u\cdot(-1)+\int_{\Gamma_i \cap T'} hu \\
    &\leq \int_{\p T' \cap \Gamma_i} u (-\nu_t) \cdot \nu_{\p T'}+ \int_{\p \Gamma_i \cap T'} u\nu_t \cdot \nu_{\p \Gamma_i}+\int_{\Gamma_i \cap T'} hu\\
    &=\int_{T' \cap \Gamma_i} \dive^{\Sigma}(u\nu_t)+hu\\
    &=\int_{T' \cap \Gamma_i} u\dive^{\Sigma}(\nu_t)+\nabla_{\nu_t} u+hu\\
    &=\int_{T' \cap \Gamma_i} u(\dive^{\Sigma}(\nu_t)+\frac{\nabla_{\nu_t} u}{u}+h)\leq 0
\end{align*}
in the final inequality we used (\ref{strongmaxi}). Indeed, as $t \rightarrow 0$, $\nu_t \rightarrow -\nu$,  $\dive(\nu_t) \rightarrow -k_{\p\Sigma}$, so by (\ref{strongmaxi}), $$h-\frac{\nabla_{-\nu_t} u}{u}+\dive(\nu_t)<0.$$  
So any minimizing sequence must be disjoint from $T'$ provided we choose $\epsilon$ small enough. 

We now write out the second variation formula for a minimizer of $\mathcal{F}(\Gamma)$. This was derived in Theorem 6.3 of \cite{wu2023free} (see also \cite{chodosh2024generalized} Lemma 16),
\begin{align*}
    0 &\leq \int_{\gamma} |\nabla_{\gamma} \psi|^2 u -\frac{1}{2}R_{\Sigma}\psi^2 u-\frac{1}{2}k^2_{\gamma}\psi^2 u +(\Delta_{\Sigma}u-\Delta_{\gamma} u)\psi^2 -\langle\nabla_{\Sigma}u, \nu_{\gamma}\rangle \psi^2 h -\langle\nabla_{\Sigma }h, \nu_{\gamma} \rangle \psi^2 u \nonumber\\
    &\leq  \int_{\gamma} |\nabla_{\gamma} \psi|^2 u - (\Delta_{\gamma}u)\psi^2 -(\frac{h^2}{2}+\nabla_{\nu_{\gamma}}h)\psi^2 u \nonumber+\frac{\psi^2(\nabla_{\gamma} u)^2}{2u}\\
    &< \int_{\gamma} |\nabla_{\gamma} \psi|^2 u +\nabla_{\gamma} \psi^2 \nabla_{\gamma} u +\frac{\psi^2(\nabla_{\gamma} u)^2}{2u} \quad (\star )
\end{align*}
where we used (\ref{PSCu}) in the second inequality and also $\frac{k_{\gamma}^2}{2}=\frac{h^2}{2}+\frac{(\nabla_{\nu_{\gamma}}u)^2}{2u}-\frac{h\nabla_{\nu_{\gamma}}u}{u}$. To get the strict inequality in ($\star$) we used $\frac{1}{2}h^2+\nabla_{\nu_{\gamma}}h >0$. 
Then we can plug  $\psi =u^{-\frac{1}{2}}$ into ($\star$) to get $0<\int_{\gamma}\frac{-1}{4u^2}(\nabla_{\gamma}u)^2$, a contradiction. So $d_{\Sigma}(x,\p\Sigma)\leq \frac{2}{a_0}$ as claimed.
\end{proof}

\begin{remark}\label{RMchange}
    One can check that in the above proof, it's sufficient to have equation (\ref{PSCu}): $\Delta_{\Sigma} u \leq \frac{R_{\Sigma}}{2} u +\frac{|\nabla_{\Sigma}u|^2}{2u}$ over the set $U:=\{y\in \Sigma, d(y,\p \Sigma)<\frac{2}{a_0}+3\delta\}$ instead of requiring it everywhere in $\Sigma$, because (\ref{PSCu}) is only used in the step before ($\star$) over the minimizer $\gamma$, which must lie in $U$ by construction.    
\end{remark}

Before proving Theorem \ref{3d-decomp}, we start with the simple case which allows us to find capillary surfaces in a compact manifold with sufficiently large boundary diameter.

\begin{lemma}\label{3d1bubble}
    If $(V,\p V, g)$ is a connected compact 3-manifold with $R_V \geq 0$, the mean curvature of the boundary satisfies $ H_{\p V}\geq \frac{\pi}{d_0}+a_0$ for some $a_0>0$ , and $\p V$ is connected with intrinsic $\text{diam}(\p V) >d_0>0$. Then there are finitely many capillary surfaces $(\Sigma_i)_{i=1}^k$ with nonempty boundary of bounded length:  $|\p \Sigma_i| \leq \frac{2\pi}{a_0}$, and $d_{\Sigma_i}(x,\p \Sigma_i)\leq \frac{2}{a_0}$ for any $x\in \Sigma_i$. Furthermore, each $\Sigma_i$ is a topological disc, and $\cup_{i=1}^k \Sigma_i$  separates $\p V$.
\end{lemma}

\begin{proof}
    Let $\text{diam}(\p V)=d_{\p V}(p,q)>d_0 + 5 \delta$, for some $\delta>0$ and $p,q \in \p V$. Then we can build a smooth function $w:\p V \rightarrow \R$ such that $w(x)=-1$ if $d_{\p V}(p,x)\leq 2\delta$ and $w(x)=1$ if $d_{\p V}(p,x)\geq d_0+3\delta$, and when $-1<w(x)<1$, then $w(x)=\cos \rho(x)$ with $|\nabla^{\p V} \rho|<\frac{\pi}{d_0}$.

    We consider a minimizer of the following functional  among open sets $\Omega$ with finite perimeters, containing $B^{\p V}_{2\delta}(p)$ and disjoint from the set $\p V \setminus B^{\p V}_{d_0+3\delta}(p)$,
\begin{align*}
    \mathcal{A}(\Omega)=|\p \Omega|+\int_{\p V \cap \Omega} w.
\end{align*} 
\textit{Claim:}
Any minimizing sequence $\Omega_i$  must (eventually) contain a fixed open neighborhood around $W_-:=\{x \in \p V:w(x)=-1\}$ and (eventually) be disjoint from some fixed open neighborhood of $W_+:=\{x \in \p V:w(x)=1\}$; the boundary  of $\p \Omega$, i.e. $\p \Omega \cap \p V$, of a minimizer $\Omega$, must lie in the set $\{x\in \p V: |w(x)|< 1\}$. We remark that the regularity for capillary surfaces (see \cite{philippis2015regularity},\cite{chodosh2024improved}) requires that $\p \Omega \cap \p V   \subset\{x\in \p V: |w(x)|< 1\}$.

The proof of these two claims is analogous to Step 5 of Theorem \ref{2d-cpt}. Let $\overline{N}$ be the outward unit normal of $\p V \subset V$, we consider the following family $\Phi_t(x):=\exp(-\varphi_t(x)\overline{N}(x))$. Here $\varphi_t(x):=\max\{s\phi(x)+t,0\}$ for some small $s>0$ to be chosen, and $\phi:\p V \rightarrow [-1,1]$ is a smooth function, such that $\{x\in \p V: \phi(x)>0\}=W_-^{\circ}$, and $\nabla \phi(x)\neq 0$ for any $x\in \p W_-$.
The same argument as in Step 5 of Theorem \ref{2d-cpt} shows that 
each slice $\Sigma_t:= \overline{\Phi_t( \p V) \setminus \p V}$ is a smooth surface with unit normal $\nu_t$ and mean curvature $H_t\geq 1.5$ and $\nu_t \cdot \overline{N}(x)\geq w(x)$ also holds for $x\in \Sigma_t \cap \p V$ for small $t>0$.
We consider the foliation  
$$\Omega'=\cup_{t\in [-s,s']} \Sigma_t$$ 
 
Then the same computation shows that,
\begin{align*}
    \mathcal{A}(\Omega'\cup \Omega)-\mathcal{A}(\Omega)&\leq\int_{\p \Omega' \setminus \Omega} \nu_t \cdot \nu_{\p \Omega'}-\int_{\p \Omega \cap \Omega'}\nu_t\cdot \nu_{\p \Omega} +\int_{\p V \cap \left(\Omega'\setminus\Omega\right)}w\\
    &=\int_{\Omega'\setminus\Omega}\dive_{\Sigma_t}(\nu_t)+\int_{\p V \cap \left(\Omega'\setminus\Omega\right)} \nu_t \cdot (-\overline{N}) + w\\
    &\leq \int_{\Omega'\setminus\Omega}\dive_{\Sigma_t}(\nu_t) \leq 0.
\end{align*}

We finished the proof of the claim.
So we find a minimizing Caccioppoli set $\Omega$ with smooth boundary $\Sigma=\p \Omega \neq \emptyset$ since $V$ is connected, and $\p \Sigma \neq \emptyset$ since $\p V$ is connected.  
We note that $\p \Omega$ many have many components, some of which are closed surface with no boundary. We want to examine each component $\Sigma$ with non-empty boundary below.

First we want to apply Lemma \ref{DistBound} to get the distance bound. We write $\sqrt{1-w^2}=\sin \rho \neq 0$.
We have the first and second variation formulas over $\Sigma$ (\cite{ros1997stability}); here we also used the Gauss-Codazzi equation $R_V=R_{\Sigma}+2\Ric(N,N)+|\sff_{\Sigma}|^2-H_{\Sigma}^2$ in the second variation: 
\begin{align}
    H_{\Sigma}&=0, \quad \langle \nu,\overline{\nu} \rangle=-w;\nonumber\\
    0&\leq \int_{\Sigma} -\phi \Delta_{\Sigma}\phi-\frac{1}{2}({R_V}-R_{\Sigma}+|\sff_{\Sigma}|^2+H_{\Sigma}^2)\phi^2\nonumber\\
    &+\int_{\p\Sigma}\frac{\phi^2}{\sin\rho} (\langle \nabla_{\bar{\nu}}\bar{\nu},\bar{N}\rangle+\nabla_{\bar{\nu}} \rho +\cos\rho\langle N,\nabla_{\nu}\nu \rangle)+\phi \nabla_{\nu}\phi. \label{2VarCapi}
\end{align} 
where $\overline{\nu}$ is the outward unit normal of $\p \Sigma  \subset \left(\overline{\Omega} \cap \p V\right)$, $\nu$ is the outward unit normal of $ \p \Sigma\subset \Sigma$, $N$ is the outward unit normal of $\Sigma \subset \overline{\Omega}$.

Now by (\ref{2VarCapi}) there is a smooth $u>0$ over $\Sigma$ with, 
\begin{align}
    \Delta_{\Sigma}u+\frac{1}{2}(R_V-R_{\Sigma}+|\sff_{\Sigma}|+H_{\Sigma }^2)u&\leq 0,  \label{lemmaEq}\\ 
    \nabla_{\nu} u+\frac{u}{\sin \rho}[\langle \nabla_{\bar{\nu}} \bar{\nu},\bar{N}\rangle+\nabla_{\bar{\nu}}\rho+\cos \rho \langle \nabla_{\nu}\nu,N \rangle]&=0, \quad \text{along } \p \Sigma\label{lemmaBC}
\end{align}        

The following computation appeared in equation (3.8) in Li's paper \cite{li2020polyhedron} is very helpful, here $\gamma'$ is a unit tangent vector along $\p \Sigma$,  
\begin{equation}\label{Li-capillary}
    -H_{\p V}=\langle \nabla_{\bar{\nu}}\bar{\nu},\bar{N} \rangle +\cos\rho \langle N,\nabla_{\nu}\nu\rangle+\sin\rho\langle \nabla_{\gamma'}\gamma',\nu\rangle.
\end{equation}

Combining equation (\ref{lemmaEq}),(\ref{lemmaBC}),(\ref{Li-capillary}) and $R_V \geq 0$,  we have a smooth $u>0$, 

\begin{align} \label{specPSC}
\begin{split}
	\Delta_{\Sigma}u &\leq \frac{1}{2}R_{\Sigma}u\\ 
    \frac{\nabla_{\nu}u}{u}&=\frac{1}{\sin \rho}(H_{\p V}-\nabla_{\overline{\nu}}\rho)+\langle \nabla_{\gamma'}\gamma',\nu\rangle\geq \frac{\frac{\pi}{d_0}+a_0}{\sin \rho}-\frac{\pi}{d_0}-k_{\p \Sigma}\geq a_0-k_{\p \Sigma}. 
\end{split}
\end{align}

So we can now apply Lemma \ref{DistBound} to get that $d_{\Sigma}(x,\p\Sigma)\leq \frac{2}{a_0}$ for all $x\in \Sigma$.  

Now recall the Gauss Bonnet theorem (for a surface with nonempty boundary $\chi_{\Sigma}\leq 1$),
\begin{align}
    2\pi &\geq 2\pi \chi_{\Sigma}=\int_{\Sigma}\frac{1}{2}R_{\Sigma}+\int_{\p \Sigma} k_{\p \Sigma}=\int_{\Sigma}\frac{1}{2}R_{\Sigma}-\int_{\p \Sigma} \langle \nabla_{\gamma'} \gamma',\nu \rangle, \label{GB}
\end{align}

Combining (\ref{Li-capillary}), (\ref{GB}) and the second variation (\ref{2VarCapi}), we have that for $\phi=1$,
\begin{align}\label{euler-char-bound}
\begin{split}
0&\leq \int_{\Sigma} -\frac{1}{2}(R_V+|\sff_{\Sigma}|^2)+2\pi\chi_{\Sigma}+\int_{\p \Sigma}\frac{1}{\sin \rho}(\nabla_{\bar{\nu}}\rho-H_{\p V}) 
\\&< 2\pi-|\p\Sigma|\cdot(-\frac{\pi}{d_0}+\frac{\pi}{d_0}+a_0)\\
&=2\pi -|\p \Sigma|\cdot a_0,
\end{split}
\end{align} 
note if $\chi_{\Sigma}\leq 0$, we would get a contradiction in (\ref{euler-char-bound}). 
So any $\Sigma$ with nonempty boundary, must be a topological disk. 
We now get $|\p \Sigma| \leq \frac{2\pi}{a_0}$ as desired.
\end{proof}

We can now continue to the proof of Theorem \ref{3d-decomp}. 
\begin{proof}[Proof of Theorem \ref{3d-decomp}]
We  first build one capillary surface $\Sigma=\Sigma_1$ using the same idea in Lemma \ref{3d1bubble} adapted to non-compact manifolds, then we build  $\Sigma_2, \Sigma_3, \Sigma_4 ...$ one by one.  

We consider a minimizer of the following functional  among open sets $\Omega$  with finite perimeter, containing $B^{\p M}_{2}(x_0) $ (for some fixed $x_0 \in \p M$) and contained in $B^M_{20}(x_0)$,
\begin{align*}
    \mathcal{A}(\Omega)=|\p \Omega|+\int_{\p M \cap \Omega} w,
\end{align*} 
here $w:\p M \rightarrow[-1,1]$ is the following Lipschitz function, 
\begin{equation*}
w(x) = \begin{cases}
-1 & x\in B^{\p M}_2(x_0)\\
\cos\rho(x) & -1<w(x)<1 \\
1 & x\in \p M \setminus B^{\p M}_{2\pi}(x_0)
\end{cases}
\end{equation*}

    Here   $\rho(y):\p M \rightarrow \R$ is a smooth function with$|\nabla^{\p M} \rho|\leq 1$. We denote $W_{\pm}:=\{x\in \p M: w(x)=\pm 1\}$ 
    and require the set $\{x\in \p M: \rho(x)=\pi\}=\p W_-$ and the set $\{x\in \p M: \rho(x)=0\}=\p W_+$ to be smooth submanifolds.

We want to apply the same proof of Lemma \ref{3d1bubble} to $B^M_{20}(x_0)$ so we need to perturb the metric of $M$ near $\p B^M_{20}(x_0)$ to get mean convexity. In particular, in a local coordinates near $\p B^M_{20}(x_0)$ one can change the Christoffel symbols (which are first derivatives of the metric) while maintaining the coordinates being orthonormal (which is a condition only depending on the zero-th order derivatives of the metric).
So we can perturb the metric near $\p B^M_{20}(x_0)$ so that the mean curvature is at least 1.5, and the perturbation happens within $M \setminus B^M_{19}(x_0)$. We denote the scalar curvature after perturbation as $\hat{R}_M$, and $\hat{R}_M \rvert_{B_{19}(x_0)}=R_M \rvert_{B_{19}(x_0)}$.  

Then we apply the same minimizing scheme as in Lemma \ref{3d1bubble}, note 
we again have that 
the boundary  of a minimizer must lie in $\{y\in \p M, -1<w(y)<1\}$  and in this region $w(y)$ can be written as $w=\cos(\rho(y))$ for some smooth function $\rho$ with $|\nabla^{\p M} \rho| \leq 1$. 
So we find a smooth minimizer consisting of finitely many capillary surfaces, we now analyze each component $\Sigma$ that has non-empty boundary.


Similar to (\ref{specPSC}) (here put $a_0=1,d_0=\pi, H_{\p M}\geq 2, |\nabla^{\p M}\rho|\leq 1$), using the second variation, we have a smooth function  $u: \Sigma \rightarrow (0,\infty)$, 
\begin{align*}
    \Delta_{\Sigma}u &\leq \frac{1}{2}(R_{\Sigma}-\hat{R}_M)u\\
    \frac{\nabla_{\nu}u}{u}&=\frac{1}{\sin \rho}(H_{\p M}-\nabla_{\overline{\nu}}\rho)+\langle \nabla_{\gamma'}\gamma',\nu\rangle\geq1-k_{\p \Sigma},
\end{align*}
where $\overline{\nu}$ is the outward unit normal of $\p \Sigma  \subset \left(\overline{\Omega} \cap \p M\right)$, $\nu$ is the outward unit normal of $ \p \Sigma\subset \Sigma$.

Using $\hat{R}_M \rvert_{B_{19}(x_0)}=R_M \rvert_{B_{19}(x_0)} \geq 0$
we have, 
$$ \Delta_{\Sigma}u\leq \frac{1}{2}(R_{\Sigma}-\hat{R}_M)u\leq \frac{1}{2}R_{\Sigma}u \quad \text{over the set } B_{19}(x_0).$$  
Using Remark \ref{RMchange} to the set $U:=\{y\in \Sigma,d(y,\p\Sigma)<\frac{2}{1}+3\}\subset B_{19}(x_0)$, we can now apply Lemma \ref{DistBound} to get that $d(x,\p\Sigma)\leq 2$ for all $x\in \Sigma$. So $\Sigma\subset B_{19}(x_0)$, $\hat{R}_M \rvert_{\Sigma}= R_M \rvert_{\Sigma}$, and in the analysis below we just write $R_M$. 

Similar to (\ref{euler-char-bound}) we have,
\begin{align*}
0\leq \int_{\Sigma} -\frac{1}{2}(R_M+|\sff_{\Sigma}|^2)+2\pi\chi_{\Sigma}+\int_{\p \Sigma}\frac{1}{\sin \rho}(\nabla_{\bar{\nu}}\rho-H_{\p M}) \leq 2\pi-|\p\Sigma|,
\end{align*} 
here again $\chi_{\Sigma}\leq 1$ since there is at least one boundary component; if $\chi_{\Sigma}\leq 0$, we would get a contradiction in (\ref{euler-char-bound}). So any $\Sigma$ with nonempty boundary, must be a topological disk, and we have $|\p \Sigma| \leq 2\pi$ as desired.

In total, we have been able to construct finitely many capillary disks, we denote $\Sigma_1$ as the union of all components of $\p \Omega$ that intersect the boundary $\p M$, $Z_1=\Omega \cap \p M$, and each component $\Sigma^{\alpha}_1$ of $\Sigma_1$ have the following, 
\begin{align*}
    B^{\p M}_2(x) \subset Z_1, \quad &\p Z_1 =\p \Sigma_1 \subset B^{\p M}_{2\pi}(x)\setminus B^{\p M}_{\pi}(x),\\ |\p \Sigma^{\alpha}_1|\leq 2\pi, \quad &d(y,\p \Sigma_1)\leq 2 \text{ for all } y \in \Sigma_1.
\end{align*}

For the non-compact boundary $\p M$, to get our desired exhaustion, we can replace the $w$ function above by $w_k$, such that $w_k$  is a mollification of $w_k=-\cos \rho_k$ for $\rho_k$ a mollification of $\tilde{\rho}_k=k\pi-d_{\p M}(x,\cdot)$ on $B^{\p M}_{(k+1)\pi}(x) \setminus B^{\p M}_{k\pi}(x) $, and $|w_k|=1$ elsewhere. For the corresponding $\mathcal{A}_k(\Omega)$, we can obtain a minimizer $\Omega_k$, $\p \Omega_k$ has finitely many components. We write $Z_k=\Omega_k \cap \p M$.  For each component $\Sigma^{\alpha}_k$ of $ \Sigma_k$ that intersects the boundary $\p M$, we have 
\begin{align*}
    B^{\p M}_{(k-1)\pi+2}(x) \subset Z_k, \quad &\p Z_k =\p \Sigma_k \subset B^{\p M}_{(k+1)\pi}(x)\setminus B^{\p M}_{k\pi}(x),\\ |\p \Sigma^{\alpha}_k|\leq 2\pi, \quad &d(y,\p \Sigma_k)\leq 2 \text{ for all } y \in \Sigma_k.
\end{align*}      


Since $B^{\p M}_{(k-1)\pi+2}(x) \subset Z_k$, we have $\cup_k Z_k =\p M$. For any $z \in \p M \setminus Z_k$, any path from $z$ to $x$ must contain a point in $\p \Sigma_k$ by connectedness of $\p M$.  
We therefore obtained an exhaustion of $\p M$ via boundary of capillary surfaces of disk type, with length at most $2\pi$.       

If we know $\pi_1(\p M)=0$, consider $E_k$ an unbounded component of $\p M \setminus Z_k$ such that $E_{k+1} \subset E_k$.    Then as in \cite{chodosh2022complete} Proposition 3.2, for each $k$, $\p E_k$ must be connected, $\overline{E_{k}} \setminus E_{k+1}$ is also connected. 
We have $\sup_k \text{diam}_{\p M}(\overline{E_{k}} \setminus E_{k+1}) \leq 5\pi$. Indeed, consider any two points $z_1, z_2 \in \overline{E_{k}} \setminus E_{k+1}$, then for each $z_i$ there's some $y_i \in \p \Sigma_k$, such that $d_{\p M}(z_i,y_i) \leq 2\pi$. Now  $|\p \Sigma_k|\leq 2\pi$ implies $d_{\p M}(y_1,y_2)\leq \pi$, adding up the length proves the claim. 

\end{proof}

\begin{remark}
    As a concrete example, take any surface $\Sigma$  satisfies Theorem \ref{2d-cpt}, i.e. with nonnegative scalar curvature and strictly convex boundary, then Theorem \ref{3d-decomp} applies to $\Sigma \times \R$. For example if $\Sigma^2$ is a strictly convex spherical cap (see \cite{ros1997stability}), then Theorem \ref{3d-decomp} applies to any small perturbation of $\Sigma^2 \times \R$.   We note that if $M^3$ has nonnegative Ricci curvature and strictly mean convex boundary, it might not split as $\Sigma \times \R$. For example consider capping off one end of a solid cylinder by a half-ball. If $M^3$ has nonnegative Ricci curvature, and strictly positive second fundamental form, then \cite{li2014sharp} has conjectured that it must be compact. 
\end{remark}

\begin{remark}
    In the proof above the bound $d(x,\p \Sigma)\leq 2$ for all $x\in \Sigma$ will depend on mean curvature of $\p M \subset M$ and also on the angle function $w$ prescribed on the boundary.  In \cite{ros1997stability}, the authors showed the disks and spherical caps are capillary stable surfaces of constant mean curvature in $\B^3$  when $w=\cos \theta$ is a constant function over $\p \B^3$.     
\end{remark}

\begin{remark}
    Theorem \ref{3d-decomp} only describes the behavior of $\p M$, and has no control of the interior of $M$. As an example, if $B=\B^3_r(z)$ is a small interior ball in $M^3$ with $R_M \rvert_{B} >0$, then one can concatenate $M\setminus B$ along $\p B \approx \Sph^2$  with $\Sph^2 \times [0, \infty)$ or to any complete (non-compact) manifold $(N,\p N)$ with $  \p N $ homeomorphic to $\Sph^2$, and $R_N \rvert_{\p N} >0$.  In particular, there is no equivalent of Corollary \ref{3d-growth} even if one assumes $M$ is simply connected. Endowing $\Sph^2 \times [0,\infty)$ with the spatial Schwarzschild metric as one interior end of $M$, the volume growth is Euclidean, instead of linear.
\end{remark}

We can now apply the same method to prove the bound for 1-Urysohn width if a simply connected surface can be filled in by a 3-manifold with a metric with  strict mean convexity and nonnegative scalar curvature.

\begin{definition}[1-Urysohn width, \cite{gromov2019four},\cite{guth2017volumes}]
    Let X be a compact metric space, we say that X has Uryson 1-width bounded by $L$, if there is a graph $G$ (a 1 dimensional simplecial complex) and a continuous map $f:X\rightarrow G$, such that every fiber $f^{-1}(y)$ for $y\in G$ has diameter bounded by $L$.
\end{definition}

\begin{proof}[Proof of Theorem \ref{fillin}]
There is some $n\in \N$ so that $nd_0<\text{diam}(\p N)\leq (n+1)d_0$. If $n\leq 1$ then we are done, otherwise let $d_{\p N}(p,q)=\text{diam}(\p N)=nd_0+5\delta$, then  we can obtain functions $(w_l)_{l=1}^{n}: \p N \rightarrow \R$, $w_l(y)=-1$ if $d_{\p N}(p,y)\leq 2\delta+(l-1)d_0$ and $w_l(y)=1$ if $d_{\p N}(p,y)\geq ld_0+3\delta$. 
When $-1<w_l(y)<1$ then $w_l=\cos\rho_l(x)$ for $|\nabla^{\p N} \rho_l(y)|\leq \frac{\pi}{d_0}$ for any $y\in \p N$.

We consider a minimizer of the following functional  among open sets $\Omega$ with finite perimeters, containing $S_l:=B^{\p N}_{2\delta+(l-1)d_0}(p)$ and disjoint from the set $S_l':=\p N \setminus B^{\p N}_{ld_0+3\delta}(p)$,
\begin{align*}
    \mathcal{A}(\Omega)=|\p \Omega|+\int_{\p N \cap \Omega} w_l.
\end{align*} 
Then we get (using the proof of Lemma \ref{3d1bubble}) $\Sigma_l =\p \Omega_l$ capillary surfaces with finitely many disk components $(\Sigma_l^k)_{k=1}^K$  and the boundary  $\p \Sigma_l^k$ of each component $\Sigma_l^k$ must separate $\p N$ into two components using simply connectedness of $\p N$. We  also have $|\p \Sigma_l^k| \leq \frac{2\pi}{a_0}$.  

Now $\p M \setminus \left(\cup_{l=1}^n \Sigma_l\right)$ is a union of finitely many  2-manifolds $\left(\Gamma_l^j\right)_{l=0}^n \subset \left(\overline{\Omega_{l+1}}\setminus \Omega_l \right)$, if $l=0$ we think of $\Omega_l$ as the empty set, if $l=n$ we think of $\Omega_{l+1}$ as $N$. Each component $\Gamma_l^j$ has piecewise smooth boundary, such that $\p \Gamma_l^j \cap \p \Sigma_l$ has at most one component (exactly 1 if $l\geq 1$) by simply connectedness.   Indeed, if not then assume that there are $\p \Sigma_l^1,\p \Sigma_l^2$ two distinct boundary components in $\Gamma_l^j$, then take a fixed point $z \in \Gamma_l^j$ and distance-minimizing paths $l_s$  to $\p \Sigma_l^s$ for $s\in\{1, 2\}$, concatenate $l_1, l_2$ together with a path in $\Omega_l$  transversal to both $l_1, l_2$ (for example by minimize distance to the point $p$), we get a contradiction in terms of intersection number with respect to each of $l_1, l_2$ using simply connectedness.

We show that each $\Gamma_l^j$ has  diameter bounded by $4d_0 +\frac{\pi}{a_0}$. If $l=0$ or $l=n$ this is true by triangle inequality. We now look at $0<l<n$, for any $z_1, z_2 \in \Gamma_l^j$, take the distance minimizing path to $\p \Gamma_l^j \cap \p \Sigma_l$ called  $t_1, t_2$ respectively.   The length of each $t_1, t_2$ is no more than $2d_0$ by construction. We have just shown that $\p \Gamma_l^j \cap \p \Sigma_l$  is a connected curve with length no more than $\frac{2\pi}{a_0}$,  so the distance between any two points on the curve is no more than $\frac{\pi}{a_0}$. Adding these together we get the bound of diameter.

Using $\text{diam}(\p \Sigma_l^k)\leq \frac{\pi}{a_0}$ for each $l,k$, we can find some tubular neighborhood $U_l^k$  of $\p \Sigma_l^k$ in $\p M$ so that $\text{diam}(U_l^k)\leq 4d_0+\frac{\pi}{a_0}$.        

We now define a graph $G$  and a continuous map $l:\p N \rightarrow G$. The graph $G$ has vertices $v_l^j$ and $l(x)=v_l^j$ if $x\in \Gamma_l^j \setminus \left(\cup_{lk}  U_l^k\right)$. We connect two vertices $v_l^j$ and $v_{l+1}^{j'}$ with an edge $E_{lj'}=[0,1]$ if $\Gamma_l^j$ and $ \Gamma_{l+1}^{j'}$ are separated by some $\p \Sigma_{l+1}^k$. For a point $z$ in the  tubular neighborhood $U_{l+1}^k$ homeomorphic to $$  \p \Sigma_{l+1}^k \times [0,1]=\{(y,t),y\in \p \Sigma_{l+1}^k, t \in [0, 1]\},$$ we map by $l(z)=l(y,t)=t\in E_{lj'}$.

One can check that this gives us a continuous map of $\p N$ to a connected graph, the preimage of every point has diameter bounded by $4d_0+\frac{\pi}{a_0}$. Using the definition of Urysohn width we have shown that the 1-Urysohn width of $\p N$ is bounded by $4d_0+\frac{\pi}{a_0}$.  
\end{proof}

\begin{remark}
    In Theorem \ref{fillin} these quantities  $R_g \geq 0, H_{\p N} \geq \frac{\pi}{d_0}+a_0, U_1(\p N) \leq 4d_0+\frac{\pi}{a_0}$ scale accordingly. One can check that the minimum of $$(\frac{\pi}{d_0}+a_0)(4d_0+\frac{\pi}{a_0})=5\pi+\frac{\pi^2}{a_0d_0}+4a_0d_0,$$  is obtained when   $a_0d_0=\frac{\pi}{2}$. So we can restate the theorem with $R_g\geq 0, H_{\p N}\geq 3a_0$, and $U_1(\p N) \leq 3\pi /a_0$.  
\end{remark}

\section{Band Width Estimate}\label{bandwidth}
We now continue with the proof of the band width estimate.

\begin{proof} [Proof of Theorem \ref{bandwidthest}] See Figure 1 (in introduction) as an example. We argue by contradiction.
If $d_{M}(S_+,S_-) > 4+\pi+2\delta$ for some $\delta>0$, then there is a smooth function  $w: \p_0 M \rightarrow \R$, such that $|w(x)|=1$ if $x\in \p_0 M \setminus K$,  $w(x)=\cos \rho(x)$ if $x\in K$, for a compact set $K$ with boundary in $\p M_0$: 
$$ (K,\p_{\pm} K) \Subset K_0 := \left\{x \in \p_0 M, 2+\delta<d_{\p M}(x,S_-)<2+\pi+2\delta\right\},$$ 
and $\rho$ is  a smooth function with
$$\rho\rvert_{\p_- K}=\pi, \rho\rvert_{\p_+ K}=0, |\nabla^{\p M} \rho|<1.$$ 
We may further assume $|w|=1$ on $\p M \setminus K_0$, in our choice we have $w \rvert_{S_-}<-1<0, w \rvert_{S_+}>1>0 $. 

Now we minimize the following functional over open Caccioppoli sets $\Omega$, such that $\p_- M \subset  \Omega, \p_+ M \cap \Omega=\emptyset$: 
\begin{equation}\label{BWfunctional}
    \mathcal{A}(\Omega)=|\p \Omega|+\int_{\p M \cap \Omega} w.
\end{equation}
Using the same argument in section \ref{3d}, we know any minimizing sequence has its boundary contained in the compact set $K$. 
By our convexity assumption, we note that if the minimizing Caccioppoli set has an interior point on $\p_{\pm} M$, then by the maximum principle (\cite{solomon1989strong}), it must be equal to a component of $\p M$ that is minimal in $M$, a contradiction to our assumption.

By regularity of capillary problem \cite{philippis2015regularity} \cite{chodosh2024improved}, we have  a smooth minimizer $\Sigma'= \p \Omega'$.  
 
By the same proof as in Lemma \ref{3d1bubble} (plug in $d_0=\pi,a_0=1$), for each component $\Sigma$ of $\Sigma'$, we have  $d_{\Sigma}(x,\p \Sigma) \leq 2$ for $x\in \Sigma$, this gives us that $d_M(\Sigma, S_{\pm})>0$. So we have that $M_- \subset \Omega'$, $\Omega' \cap M_+ =\emptyset$, $\p \Omega'=\Sigma'\subset M_0$ and we can find a map of degree 1 (mod 2) from a component $\Sigma$ of $\Sigma'$  onto $S_-$ as follows. 

Consider a diffeomorphism $f: M_0 \rightarrow S_- \times [-1,1]$, then for any $x=f^{-1}(z_0,t_0) \in \Sigma'$, we define the projection map $P(x)=z_0$. Consider a regular value  $z \in S_-$ and the path $I_z:= f^{-1}(z,t)$ in $M_0$ ,  we may write $I_z$ just as $(z,t)$ for $t \in [-1,1]$. We look at the characteristic function of $\Omega'$,  then $\chi_{\Omega'}(z,-1)=1$ and $\chi_{\Omega'}(z,1)=0$, we note that $z$ can be in $\p S_-$, and $P^{-1}(\p S_-)=\p \Omega'$, so there is some component $\Sigma$ of $\Sigma'$ such that the degree of $P$ is equal to 1 (mod 2). 
We also note that this means for any $x\in \p S_-$, $P^{-1}(x) \cap \Sigma \neq \emptyset$. If $\p S_-$ has at least two components, then $\chi(\Sigma)\leq 0$, otherwise we have a map of degree 1 (mod 2) from $(\Sigma, \p \Sigma) \rightarrow (S_-,\p S_-)$.      

Apply Kneser's theorem \cite{kneser1930kleinste} to the double of $\Sigma$ and $S_-$, we have $\chi(\Sigma) \leq 0$. Therefore, using the stability inequality (see also \ref{euler-char-bound}) we get, 
\begin{align*}
0\leq& \int_{\Sigma} -\frac{1}{2}(R_M+|\sff_{\Sigma}|^2)+\int_{\p \Sigma}\frac{1}{\sin \rho}(\nabla_{\bar{\nu}}\rho-H_{\p M}) +2\pi\chi_{\Sigma}\\
< & \int_{\p \Sigma} \frac{1}{\sin \rho}(1-2)+2\pi\chi_{\Sigma}<0
\end{align*} 
a contradiction. 
\end{proof}

We first give some examples related to Theorem \ref{bandwidthest}.
\begin{example}
We give an example when $M_0$ is homeomorphic to $  S_- \times [-1,1]$ with $\chi(S_-)> 0$, to show that in Theorem \ref{bandwidthest}, the assumption of $\chi(S_-) \leq 0$ is crucial. Take the cylinder $\D^2 \times [-L,L] \subset \R^3$, with $\D^2$ unit disk in $\R^2$, and cap the top (or bottom) slice $\D^2 \times \{\pm L\}$ with an upper (or lower) hemisphere, then the length of the cylinder is unbounded when we let $L \rightarrow \infty$.     
\end{example}


\begin{remark}
    For $(M,\p M)$ as in Theorem \ref{bandwidthest}, following the proof, we can replace the statement $R_M \geq 0, H_{0} \geq 2, H_{\pm} \geq 0, d_{M}(S_+,S_-) \leq 4+\pi$, with the following:
    if $R_{M_0} \geq 0, H_{0} \geq \frac{\pi}{d_0}+a_0, H_{\pm} \geq 0$, then $  d_{M}(S_+,S_-) \leq d_0+\frac{4}{a_0}$.
\end{remark}

\printbibliography

\end{document}